\documentclass[a4paper,12pt]{article}

\usepackage{amsmath ,amsfonts,  amssymb, amscd, amsthm} 
\usepackage{bbm}
\usepackage{mathrsfs}
\usepackage[font=small,skip=0pt]{caption}

\usepackage{csquotes, color}
\usepackage{colortbl}
\usepackage[numbers]{natbib}
\usepackage{verbatim}
\usepackage[toc,page]{appendix}
\usepackage{caption}
\usepackage{textcomp}
\usepackage[inline]{enumitem}
\usepackage{tcolorbox}
\usepackage{cancel}
\usepackage{relsize}
\usepackage{bm}

\usepackage{authblk}

\newtheorem{theorem}{Theorem}
\newtheorem*{theorem*}{Theorem}
\newtheorem{corollary}{Corollary}[theorem]
\newtheorem{lemma}{Lemma}
\newtheorem{proposition}{Proposition}

\newtheorem*{definition*}{Definition}

\newtheorem*{example*}{Example}
\newtheorem{remark}{Remark}
\newtheorem*{conditional*}{Conditional}

\newtheorem*{remark*}{Remark}

\definecolor{Gray}{gray}{0.85}
\definecolor{LightCyan}{rgb}{0.88,1,1}

\newcolumntype{a}{>{\columncolor{Gray}}c}
\newcolumntype{b}{>{\columncolor{Gray}}r}
\newcolumntype{d}{>{\columncolor{white}}c}

\usepackage{hyperref}
\hypersetup{
    colorlinks=true,
    citecolor=blue,
    linkcolor=blue,
    filecolor=cyan,      
    urlcolor=blue,
}
 
\providecommand{\keywords}[1]
{
  \small	
  \textbf{\textit{Keywords---}} #1
}

\title{
Distributions based on Stable Mixtures and  Gamma-Stable Convolutions}
\author{Nomvelo Karabo Sibisi \\{\small {\tt sbsnom005@myuct.ac.za}}}
\date{\today}

\begin{document}
\maketitle
\thispagestyle{empty}


\begin{abstract}
\noindent
This  paper explores mixture distributions induced by a product of the positive  stable random variable
and a  power of another positive random variable. 
The paper also considers the convolution of the stable density with a gamma density. 
These two constructs, mixing and convolution, suffice to generate a rich family of distributions.
An example  is the positive Linnik distribution, which is known to arise from a  product involving stable and gamma random variables.
We show that  gamma-Linnik convolution gives the Mittag-Leffler distribution and 
the Mittag-Leffler Markov chain associated with the growth of random trees.
Building on  that, we construct a new family of distributions with explicit densities.
A particular choice of parameters for this family yields the Lamperti-type laws
 associated with occupation times for  Markov processes.
\end{abstract}
\keywords{convolution, mixing; 
gamma, beta, stable,  Mittag-Leffler distributions; \\ Mittag-Leffler Markov chain;
positive Linnik law, Lamperti-type laws
}
 

\section{Introduction}
\label{sec:introduction}

All results derived in this paper follow from  two theorems that involve  mixing and convolution of distributions,
with the positive stable distribution as a common factor.
\paragraph{First  Theorem:}
Let $S_{\alpha;z}\sim F_\alpha(\cdot \vert z)$  $(0 <\alpha<1$, $z>0)$ 
denote a  positive stable random variable 
where $F_\alpha(t \vert z)$ $(t>0)$, with  density $f_\alpha(t \vert z)$,
is characterised by the Laplace transform
\begin{align}
\mathbb{E}\left[e^{-xS_{\alpha;z}}\right] &\equiv \int_0^\infty e^{-xt} \,  f_\alpha(t\vert z) \, dt = e^{-zx^\alpha} \qquad (x\ge0)
\label{eq:intro_stableLT}
\end{align}
Let $U\sim F$ be another   positive random variable with probability distribution $F$. 

{\bf Theorem~\ref{thm:gen_rvproduct}} states that $T_{\alpha,z}=S_{\alpha;z} \, U^{1/\alpha}$  has a distribution 
$H_\alpha(\cdot \vert z)$ with explicit density
\begin{align}
h_\alpha(t \vert z)
&=  \int_0^\infty  f_\alpha(t \vert z u) \, dF(u)  \quad (t>0)
\label{eq:intro_rvproduct_density}
\end{align}
and  Laplace transform 
$\mathbb{E}\left[e^{-xT_{\alpha,z}}\right] =\mathbb{E}\left[e^{-zx^\alpha U}\right]$ $(x\ge0)$.
The mixture density~(\ref{eq:intro_rvproduct_density}) may  be regarded  as an example of 
subordination  for distributions on $(0,\infty)$ (Feller~\cite[XIII.7({$e$})]{Feller2}).

{\bf Corollary~\ref{cor:hPollardinfsum}} uses a  representation due to  Pollard~\cite{Pollard}  of the stable density  $f_\alpha(t \vert z u)$ 
 to show that,  if $F$ has   moments $\{\mu_k\}$ of all order,
then $h_\alpha(t \vert z)$  admits the infinite series form
\begin{align}
h_\alpha(t\vert z)
 &= -\frac{1}{\pi} \sum_{k=1}^\infty \frac{(-z)^k}{k!} \sin(\pi\alpha k) \,\frac{\Gamma(\alpha k+1)}{t^{\alpha k+1}} \;  \mu_k 
\label{eq:intro_hPollardinfsum} 
\end{align}

The study of  laws induced by $S_{\alpha;z} \, U^{1/\alpha}$  is not new. 
Notably, Pakes~\cite[(4.4)]{Pakes95} considered   $U_{\nu,1}\sim G(\nu,1)$ (gamma distribution) 
and termed the  law of $S_{\alpha;z} \, U_{\nu,1}^{1/\alpha}$  the generalised positive  Linnik law.
Bondesson~\cite[3.2.4(p38)]{Bondesson} discussed  $S_{\alpha;1} U_{\nu,1}^{1/\alpha}$ in a context that  
we will  comment on  later.
We will explore various  choices of $F$,  generalised positive Linnik  distribution included.


\paragraph{Second  Theorem:}
 Let the Dirac   delta  function $\delta_c(t) \equiv\delta(t-c)$ be the   density of a jump at $t=c\ge0$.
Also let  $\rho_\nu(t) =  t^{\nu-1}/\Gamma(\nu)$ $(t>0, \nu>0)$. 

{\bf Theorem~\ref{thm:rho_stable_convolution}} states that  the  convolution $\rho_\nu \star f_\alpha(\cdot\vert z)$ may be written as 
\begin{align}
 \{\rho_\nu\star f_\alpha(\cdot\vert z) \}(t) 
 &\equiv \int_0^t \rho_\nu(t-u) f_\alpha(u\vert z) du \\
 &=  z^{\nu/\alpha}  \int_0^\infty  f_\alpha(t \vert z u)  \{\rho_{\nu/\alpha} \star \delta_1\}(u) \, du
\label{eq:intro_power_stable_conv} \\
&=  \frac{z^{\nu/\alpha}}{\Gamma(\nu/\alpha)}
  \int_0^1 f_\alpha(t\vert z/u) \, u^{-\nu/\alpha}  (1-u)^{\nu/\alpha-1}  \, \frac{du}{u} 
  \label{eq:intro_power_stable_conv_betadensity} 
\end{align}
The first integral is  the  standard Laplace convolution form.
The second integral
shows that the   convolution with the stable density may conveniently be delegated to a convolution with the 
Dirac delta function $\delta_1$ instead.
The third integral arises from the change of variable $u\to 1/u$.

If we write $f_\alpha(t)\equiv f_\alpha(t\vert 1)$,  it is easy to  see from~(\ref{eq:intro_stableLT})  that 
$f_\alpha(t\vert u) \equiv f_\alpha(t u^{-1/\alpha}) u^{-1/\alpha}$. 
Hence, for  $z=1$, $\{\rho_\nu\star f_\alpha(\cdot\vert 1) \}(t)$  in the form~(\ref{eq:intro_power_stable_conv_betadensity}) 
is   equivalent  to  Ho~{\it  et al.}~\cite[Proposition~2.2($iii$)]{HoJamesLau}, derived therein by a different approach 
with $\rho_\nu\star f_\alpha$ expressed as  the  Riemann-Liouville fractional integral  $I_+^\nu\, f_\alpha$.
It is key to note that, in our case,  we  regard $\{\rho_\nu\star f_\alpha(\cdot\vert z) \}(t)$ as a density  in two variables $(z,t)$.
Depending on context, this may be a density in $t$ given $z$ or {\it vice versa}.
 
\subsection{Novelty and Motivation}
\label{sec:novelty}

The novelty of this paper  lies in Theorem~\ref{thm:gen_rvproduct}  and Theorem~\ref{thm:rho_stable_convolution}
as well as in   bringing them together 
to generate a broad family of distributional laws based on $S_{\alpha;z} \, U^{1/\alpha}$  for general positive
$U\sim F$, 
with explicit densities 
in every case,
along with associated infinite series representations.
The   highlight is a new family of laws  that subsumes several well-studied instances.
At a minimum, the  approach of providing explicit densities complements the common approach of specifying equalities in law 
of various products of  independent random variables without  necessarily offering explicit forms of the densities involved.

The inspiration to study the family of laws described by Theorems~\ref{thm:gen_rvproduct}  and~\ref{thm:rho_stable_convolution}
stems from the   Mittag-Leffler 
distributions, which arise in diverse areas of application. 
Prominently, they arise  in the study of  occupation times for  Markov processes.
The   Mittag-Leffler distribution, often   denoted by ${\rm ML}(\alpha)$,   arises as 
the limiting distribution  for the occupation time of a single-state (Darling and Kac~{\cite{DarlingKac}),
while  multi-state cases lead to generalised arcsine laws
(Bertoin~{\it et al.}~\cite{Bertoin}, Feller~\cite[XIV.3]{Feller2}, James~\cite{James_Bernoulli, James_Lamperti}, 
James~{\it et al.}~\cite{JamesRoynetteYor},
Lamperti~\cite{Lamperti},  Pitman~\cite{Pitman_CSP, Pitman2018}).

In a study of  random partitions of the integers,
Pitman~\cite[Theorem~3.8]{Pitman_CSP}  derived the two-parameter Mittag-Leffler   distribution ${\rm ML(\alpha,\theta})$   $(\theta>-\alpha)$ 
as the limiting  distribution of the number of partition blocks generated by 
 the Chinese restaurant process construction (Pitman~\cite[3.1]{Pitman_CSP})  of   the
Poisson-Dirichlet process ${\rm PD(\alpha,\theta})$ of Pitman and Yor~\cite{PitmanYor}.
${\rm ML(\alpha,\theta})$ is also said to be the distribution of the $\alpha$-diversity of ${\rm PD(\alpha,\theta})$;
a term inspired by application in ecology, where it refers to species diversity in a local habitat.
A derivation of ${\rm ML(\alpha,\theta})$   due to Janson~\cite{Janson} takes  the form of  a limiting distribution of  a  P{\' o}lya urn scheme
(also see Flajolet~{\it  et al.}~\cite{Flajolet}).

In another  context, ${\rm ML(\alpha,\theta})$ distributions  and the associated Mittag-Leffler Markov chain
play a fundamental role in  modelling  the evolution of  random trees and graphs. 
In the study of stable trees (Goldschmidt and Haas~\cite{GoldschmidtHaas}, Rembart and Winkel~\cite{RembartWinkel2018, RembartWinkel2023})
and stable graphs (Goldschmidt~{\it  et al.}~\cite{ Goldschmidt22}), 
the distributions that feature  are Beta, ${\rm ML(\alpha,\theta})$, 
Dirichlet and Poisson-Dirichlet  ${\rm PD(\alpha,\theta})$.
James~\cite{James2015}, S\'{e}nizergues~\cite{Senizergues} discussed the appearance of  
${\rm ML(\alpha,\theta})$ as limit distributions  in the   growth of random graphs
 inspired by  the preferential  attachment model due to Barab\'{a}si and  Albert~\cite{BarabasiAlbert}.
 
 Our approach is to `stand back' from the technical detail of  application contexts and explore in the abstract distributional laws that obey
 Theorems~\ref{thm:gen_rvproduct} and~\ref{thm:rho_stable_convolution}.
 The Mittag-Leffler laws turn out to be a part of that class, as do  the  ``Lamperti-type laws'' that were so-named by James~\cite{James_Lamperti}.
  In all instances, the emphasis is on providing explicit densities. 

\subsection{Structure of Paper}
\label{sec:structure}

Section~\ref{sec:prelim} outlines known distributions that are building blocks of later discussion.
Section~\ref{sec:theorems} gives formal detail on  Theorem~\ref{thm:gen_rvproduct}  and Theorem~\ref{thm:rho_stable_convolution}.
Section~\ref{sec:Linnik} discusses the  positive  Linnik distribution
as an instance of  Theorem~\ref{thm:gen_rvproduct}.
Section~\ref{sec:ML_laws} shows how the  Mittag-Leffler distributions follow   from  Gamma-Linnik  convolution. 
 Section~\ref{sec:ML_laws} also uses Theorem~\ref{thm:rho_stable_convolution} to derive  the recursion between the densities of
 ${\rm ML(\alpha,\theta})$ and ${\rm ML(\alpha,\theta+1})$ that is  the basis of the well-known Mittag-Leffler Markov chain.
In turn, this leads to a compact derivation of the transition probability of the chain, demonstrating its time-homogenous nature. 
Section~\ref{sec:new_family} introduces a new family of distributions  based on 
$U\sim {\rm ML(\alpha,\theta})$ in Theorem~\ref{thm:gen_rvproduct}. 
We  provide explicit densities in all instances, as well as  corresponding infinite series representations.
The resulting laws include the Lamperti-type laws described by James~\cite{James_Lamperti} as a particular case.
Section~\ref{sec:generalisation} discusses a natural  generalisation of the known Mittag-Leffler distributions.
Section~\ref{sec:discussion} gives concluding remarks.

\section{Preliminaries}
\label{sec:prelim}

\subsection{Gamma Distribution}
\label{sec:gamma}

 The gamma distribution $G(\nu,\lambda)$  with shape   and scale parameters $\nu, \lambda>0$ and density  $\rho_{\gamma,\lambda}(t)$  is
\begin{align}
dG(t \vert \nu,\lambda) 
&\equiv  \rho_{\nu,\lambda}(t)  \, dt  
= \frac{\lambda^\nu}{\Gamma(\nu)}\,t^{\nu-1} \,e^{-\lambda t} dt  \qquad (t>0)
\label{eq:gamma} 
\end{align}
For $\lambda=0$, we  write   the unnormalised power density as $\rho_\nu(t) =  t^{\nu-1}/\Gamma(\nu)$.
The   Laplace-Stieltjes  transform of $G(\nu,\lambda)$ (or the  Laplace transform of its density $\rho_{\nu,\lambda}$) is
\begin{alignat}{3}
\int_0^\infty e^{-xt}\, dG(t \vert \nu,\lambda) 
\equiv \int_0^\infty e^{-xt}\,\rho_{\nu,\lambda}(t) \, dt &= \frac{\lambda^\nu}{(\lambda+x)^\nu} &&\quad (x\ge0)
\label{eq:gammaLT}  \\
\int_0^\infty e^{-xt}\, \rho_\nu(t)\, dt &=  x^{-\nu} &&\quad (x>0)
\label{eq:powerLT}
\end{alignat}  
We also define $\rho_{\nu=0}(t)=\delta(t)$ with Laplace transform 1.

\subsection{Beta Distribution}
\label{sec:beta}

The beta distribution ${\rm Beta}(\alpha,\beta)$  with density ${\rm beta}(t \vert \alpha,\beta)$ $t\in(0,1)$ conditioned on  $\alpha,\beta>0$  is
 \begin{align}
d\,{\rm Beta}(t \vert \alpha,\beta) \equiv  {\rm beta}(t\vert \alpha,\beta) dt 
&= \frac{\Gamma(\alpha+\beta)}{\Gamma(\alpha)\Gamma(\beta)} \, t^{\alpha-1}(1-t)^{\beta-1} dt 
\label{eq:beta} 
\end{align} 

In anticipation of later discussion,  let  $B_{\alpha,\beta}\sim {\rm Beta}(\alpha,\beta)$ and $U\sim F$ be independent  random variables 
where $F$ has a density $f(u)$, $u>0$.
The joint density of $U=u, B_{\alpha,\beta}=v$ and  the product  $U\times B_{\alpha,\beta}=t$ is
\begin{align*}
\Pr(t, u,v) = \Pr(t\vert u,v) \Pr( u) \Pr(v) = \delta(t-uv) f(u) {\rm beta}(v\vert \alpha,\beta)
\end{align*}
where $\delta(t-c)$ is the Dirac delta density centred at $c$.
Hence $\Pr(t,u)=\Pr(t\vert u)\Pr(u)$ is
\begin{align}
\Pr(t,u) &= \int_0^1 \delta(t-uv) f(u) {\rm beta}(v\vert \alpha,\beta) dv =  f(u) \times \frac{1}{u} \, {\rm beta}(t/u \vert \alpha,\beta) \quad (u>t) 
 \label{eq:jointdensity} \\
\implies \; \Pr(t) &= \int_t^\infty f(u) {\rm beta}(t/u \vert \alpha,\beta) \, \frac{du}{u}
  = \int_0^1 f(t/u) \, {\rm beta}(u\vert \alpha,\beta) \, \frac{du}{u}  
 \label{eq:productdensity}
\end{align}
The  notation $T_{\alpha,\beta}  \overset{d}{=} U B_{\alpha,\beta}$,
where $T_{\alpha,\beta}$ has  the distribution with density~{(\ref{eq:productdensity}),
 states equality in distribution of the  random variables  without giving explicit forms of  associated densities.

\subsection{Stable Distribution}
\label{sec:stable}

The one-sided  stable distribution $F_\alpha(z)$ for  ($0<\alpha<1$) and scale parameter $z>0$
with density $f_\alpha(u\vert  z)$ ($u>0$) is indirectly defined by   its Laplace-Stieltjes transform 
$($equivalently, the ordinary Laplace transform of  $f_\alpha(u\vert  z)$) over $u$
\begin{align}
e^{-z x^\alpha} &= \int_0^\infty e^{-xu}\,dF_\alpha(u\vert z) \equiv  \int_0^\infty e^{-x u}  f_\alpha(u\vert z)\, du \qquad (x\ge0)
\label{eq:stable} 
\end{align}
$f_\alpha(u\vert z)$ may  be written as $ f_\alpha(u z^{-1/\alpha})\, z^{-1/\alpha}$, 
where $f_\alpha(u)\equiv f_\alpha(u \vert 1)$.
Then  $F(u\vert z) \equiv F(z^{-1/\alpha} u \vert 1)$.
It is also convenient  to define
 $f_{\alpha=1}(u\vert z) = \delta(u-z)$ with Laplace transform $e^{-z x}$.

The stable density is only known in closed form for selected values of $\alpha$, the simplest  being for $\alpha=1/2$.
For the  general  rational  case $\alpha=l/k$ for positive integers $l,k$ with $k>l$, Penson and G\'{o}rska~\cite{PensonGorska} expressed
 $f_\alpha$ in terms of the Meijer $G$ function
 (also see James~\cite[8]{James_Lamperti}).

 Pollard~\cite{Pollard} (also see Feller~\cite[XVII.6]{Feller2}) proved that $f_\alpha(u \vert z)$ ($0<\alpha<1$)
has the   infinite series representation
 (with  modification here to include conditioning on the scale factor $z>0$)
\begin{align}
  f_\alpha(u\vert z)  &= \phantom{-} \frac{1}{\pi}\,{\rm Im}  \int_0^\infty e^{-uv} \, e^{-z (e^{-i\pi}v)^\alpha}  \, dv  
\label{eq:Pollard} \\
&= \phantom{-} \frac{1}{\pi}\,{\rm Im}
   \sum_{k=0}^\infty \frac{(-z)^k}{k!} e^{-i\pi\alpha k}  \int_0^\infty e^{-uv} \, v^{\alpha k} \, dv
\label{eq:Pollardinfsum1} \\
&= -\frac{1}{\pi}
   \sum_{k=1}^\infty \frac{(-z)^k}{k!} \sin(\pi\alpha k) \, \frac{\Gamma(\alpha k+1)}{u^{\alpha k+1}}
\label{eq:Pollardinfsum2} 
\end{align}
 
 \subsection{Mittag-Leffler Distributions}
\label{sec:ML}
 
 Pollard~\cite{PollardML}  proved that, for $0\le\alpha\le1$, the  Mittag-Leffler function $E_\alpha(-x)$ 
 (as defined in Appendix~\ref{sec:MLfunction})
may be written as
\begin{align}
E_\alpha(-x) &= \int_0^\infty e^{-x t} \, dP_\alpha(t) \equiv \int_0^\infty e^{-x t} p_\alpha(t) \, dt \quad (x\ge0)
\label{eq:ML1parLT}
\end{align}
where the probability  distribution $P_\alpha(t)$, with  density $p_\alpha(t) = f_\alpha(t^{-1/\alpha}) \, t^{-1/\alpha-1}/\alpha$,
is known as the  Mittag-Leffler   distribution.
Pollard used  complex analytic inversion of~(\ref{eq:ML1parLT})  to derive $p_\alpha(t)$.
In an alternative derivation, Feller~\cite[XIII.8]{Feller2} considered the two-dimensional Laplace  transform of the 
stable distribution   $1-F_\alpha(t u^{-1/\alpha})$. 
An alternative probabilistic derivation considers the distribution of   $S^{-\alpha}_{\alpha}$ where $S_{\alpha}$
is the $\alpha$-stable random variable (James~\cite{James_Lamperti}, Pitman~\cite{Pitman_CSP}).


The 2-parameter Mittag-Leffler  distribution $P_{\alpha,\theta}$ for  $(\theta>-\alpha)$ due to
Pitman~\cite[Theorem~3.8]{Pitman_CSP} has the  polynomially-tilted density
\begin{alignat}{5}
p_{\alpha,\theta}(t) 
&= \frac{\Gamma(1+\theta)}{\Gamma(1+\theta/\alpha)} \, t^{\theta/\alpha}\, p_\alpha(t)
&&=  \frac{\Gamma(1+\theta)}{ \Gamma(1+\theta/\alpha)} \,\frac{1}{\alpha} f_\alpha(t^{-1/\alpha}) \, t^{(\theta-1)/\alpha-1}  \quad && (\theta>-\alpha) 
 \label{eq:ML2par}   
\end{alignat}
$P_{\alpha,\theta}$  is also referred to as the generalised Mittag-Leffler distribution and alternatively 
denoted by ${\rm ML(\alpha,\theta)}$ 
(Goldschmidt and Haas~\cite{GoldschmidtHaas}, Ho~{\it  et al.}~\cite{HoJamesLau}). 
The Mittag-Leffler distribution is the case $P_{\alpha}\equiv{\rm ML}(\alpha,0)$.
It is known (James~\cite{James_Lamperti}) that,  for $x\ge0$, the Laplace transform of $p_{\alpha,\theta}(t)$ is  
\begin{alignat}{3}
\int_0^\infty e^{-x t} p_{\alpha,\theta}(t) \, dt 
&= \Gamma(1+\theta)\, E^{1+\theta/\alpha}_{\alpha,1+\theta}(-x)   \quad && (\theta>-\alpha) 
\label{eq:ML2parLT} 
\end{alignat}
$P_{\alpha,\theta}\equiv {\rm ML(\alpha,\theta)}$  has finite moments of all order, the $k^{\rm th}$ moment  ($k\ge0$) being
\begin{alignat}{3}
\mu_{\alpha,\theta;k} \equiv \int_0^\infty t^k \, dP_{\alpha,\theta}(t)
&= \frac{\Gamma(1+\theta)\Gamma(1+k+\theta/\alpha)}{\Gamma(1+\theta/\alpha)\Gamma(1+k\alpha+\theta)}  \quad && (\theta>-\alpha) 
 \label{eq:ML2parMoment} 
 \end{alignat}
 The existence of all moments  is key to  the new family of laws introduced in Section~\ref{sec:new_family}.

\subsection{Gamma-Stable Convolution}
\label{sec:convolution}

The convolution   $\{\rho_\nu\star f_\alpha(\cdot\vert z) \}(u)$ 
of the power density $\rho_\nu(u)$ and the stable density $f_\alpha(u\vert z)$  is 
\begin{align}
\{\rho_\nu\star f_\alpha(\cdot\vert z) \}(u) =  \int_0^u \rho_\nu(u-v) f_\alpha(v\vert z) dv
=  \frac{1}{\Gamma(\nu)} \int_0^u (u-v)^{\nu-1} f_\alpha(v\vert z) dv
 \label{eq:rhoconvstable} 
\end{align}
where  $0<\alpha<1, \nu>0, z>0$.
The Laplace transform of $\{\rho_\nu\star f_\alpha(\cdot\vert z) \}(u)$  over $u$ is  $x^{-\nu} e^{-z x^\alpha}$.

Differentiating the Laplace transform $e^{-z x^\alpha}$ of  $f_\alpha(u \vert z)$ gives
\begin{align}
z\alpha x^{\alpha-1} e^{-z x^\alpha} &= 
\int_0^\infty e^{-x u} \, u f_\alpha(u \vert z)\, du
= z\alpha \int_0^\infty e^{-x u} \{\rho_{1-\alpha}\star f_\alpha(\cdot\vert z)\}(u)\, du
\label{eq:stableLTderiv}
\end{align}
The second equality arises because   $x^{\alpha-1} e^{-z x^\alpha}$ is the Laplace transform of  
$\{\rho_{1-\alpha}\star f_\alpha(\cdot\vert z)\}(u)$.  
Uniqueness of Laplace transforms  implies that,  
for $\nu=1-\alpha$
\begin{align}
 \{\rho_{1-\alpha}\star f_\alpha(\cdot\vert z)\}(u) 
&= \frac{1}{\Gamma(1-\alpha)} \int_0^u (u-v)^{-\alpha} f_\alpha(u \vert z) \, dv
  =\frac{1}{ z\alpha}\,  u f_\alpha(u \vert z) 
\label{eq:stableLTderivdensity} 
\end{align}
This  well-known property  has inspired the study of generalised stable densities $f_{\alpha,m}(u)$ 
defined by  $z^m f_{\alpha,m}(u) = \alpha\{\rho_{1-\alpha}\star  f_{\alpha,m}\}(u)$ for $m>0$
(Jedidi~{\it et al.}~\cite{Jedidi},  Pakes~\cite{Pakes}, Schneider~\cite{Schneider}).

\section{Main Theorems}
\label{sec:theorems}

 \begin{theorem}
\label{thm:gen_rvproduct}

Consider two  independent  positive random variables
\begin{enumerate}
\setlength{\itemsep}{0pt}
\item $S_{\alpha;z} \sim F_\alpha(z>0)$ $(0<\alpha<1)$,  where $F_\alpha(z)$ is  the  $\alpha$-stable distribution with density $f_\alpha(\cdot\vert z)$ 
\item $U \sim F$, where $F$ is  some probability  distribution. 
\end{enumerate}
Then the product $T_{\alpha,z}=S_{\alpha;z} \, U^{1/\alpha}$  has a distribution $H_\alpha(z)$ with mixture density
\begin{align}
h_\alpha(t \vert z)
&=  \int_0^\infty  f_\alpha(t \vert z u) \, dF(u)  
\label{eq:gen_rvproduct_density}
\end{align}
with Laplace transform
 $\mathbb{E}\left[e^{-xT_{\alpha,z}}\right]=  \mathbb{E}\left[e^{-zx^\alpha U}\right]  (x\ge0)$. 
\end{theorem}
\begin{proof} [Proof of Theorem  $\ref{thm:gen_rvproduct}$]
\label{proof:gen_rvproduct}
The  density of $T_{\alpha,z}=S_{\alpha;z} \, U^{1/\alpha}\equiv t$ 
for  independent $S_{\alpha;z}\equiv x$ and $U\equiv u$, with $\Pr(S_{\alpha;z}\equiv x)=f_\alpha(x \vert z)$ is
\begin{align*}
\Pr(t) 
&= \int_0^\infty \int_0^\infty  \delta(t-xu^{1/\alpha}) \, f_\alpha(x \vert z)  dx \, dF(u) \\
&= \int_0^\infty \int_0^\infty  \delta(t-y)\,  f_\alpha(y u^{-1/\alpha} \vert z) u^{-1/\alpha}  dy \, dF(u)\\
&= \int_0^\infty  f_\alpha(t \vert z u)  f(u) du \equiv h_\alpha(t \vert z)
\end{align*}
where we have recalled  that $f_\alpha(t u^{-1/\alpha} \vert z) u^{-1/\alpha}\equiv f_\alpha(t\vert zu)$.
 The Laplace transform of~(\ref{eq:gen_rvproduct_density}) is 
\begin{align*}
\mathbb{E}\left[e^{-xT_{\alpha,z}}\right]
&= \int_0^\infty e^{-x t} \, h_{\alpha}(t \vert z) \, dt  
=  \int_0^\infty e^{-zx^\alpha u} \, dF(u) =  \mathbb{E}\left[e^{-zx^\alpha U}\right] 
  \qquad (x\ge0)
\end{align*}
thereby completing the proof.
\end{proof}

\begin{corollary}
\label{cor:hPollardinfsum}
 If $F$ 
 has  finite moments $\{\mu_k\}$ of all order,
then $h_\alpha(t \vert z)$ admits the series form
\begin{align}
h_\alpha(t\vert z)
 &= -\frac{1}{\pi} \sum_{k=1}^\infty \frac{(-z)^k}{k!} \sin(\pi\alpha k) \,\frac{\Gamma(\alpha k+1)}{t^{\alpha k+1}} \;  \mu_k 
\label{eq:hPollardinfsum} 
\end{align}
\end{corollary}

\begin{proof}[Proof of Corollary~$\ref{cor:hPollardinfsum}$]
Using the infinite series  representation~(\ref{eq:Pollardinfsum1}) of  $f_\alpha(t\vert z u)$ in~(\ref{eq:gen_rvproduct_density}) gives
\begin{align*}
h_\alpha(t\vert z)
  &= -\frac{1}{\pi}{\rm Im} \int_0^\infty  \left[\sum_{k=0}^\infty \frac{(-z u)^k}{k!} e^{-i\pi\alpha k} 
        \int_0^\infty e^{-tv} \, v^{\alpha k} \, dv \right]dF(u) \\
  &= -\frac{1}{\pi} \sum_{k=1}^\infty \frac{(-z)^k}{k!} \sin(\pi\alpha k) \, \frac{\Gamma(\alpha k+1)}{t^{\alpha k+1}} 
   \int_0^\infty u^k \, dF(u)
\end{align*}
The rightmost integral is $\mu_k$, thereby proving~(\ref{eq:hPollardinfsum}).
 \end{proof}

\begin{theorem}
\label{thm:rho_stable_convolution}
Let  $\rho_\nu(t)=t^{\nu-1}/\Gamma(\nu)$ $(\nu>0)$. 
Then
\begin{align}
\{\rho_{\nu} \star f_\alpha(\cdot\vert z)\}(t)
&=   z^{\nu/\alpha} \int_0^\infty f_\alpha(t\vert z u)\, \{\rho_{\nu/\alpha} \star \delta_1\}(u) du 
\label{eq:power_stable_conv}  \\
&=   z^{\nu/\alpha} \int_1^\infty f_\alpha(t\vert z u)\, \rho_{\nu/\alpha} (u-1) \,  du 
\label{eq:power_stable_conv1}  \\
&=  \frac{z^{\nu/\alpha}}{\Gamma(\nu/\alpha)}
  \int_0^1 f_\alpha(t\vert z/u) \, u^{-\nu/\alpha}  (1-u)^{\nu/\alpha-1}  \, \frac{du}{u} 
  \label{eq:power_stable_conv_betadensity} 
\end{align}
\end{theorem}

\begin{proof}[Proof of Theorem~$\ref{thm:rho_stable_convolution}$]
\label{proof:rho_stable_convolution}
The Laplace  transform of   $\{\rho_{\nu} \star f_\alpha(\cdot\vert z)\}(t)$ is $x^{-\nu}e^{-zx^\alpha}$
while that of the right hand side of~(\ref{eq:power_stable_conv}) is
\begin{align*}
z^{\nu/\alpha} \int_0^\infty e^{-zu x^\alpha}  \, \{\rho_{\nu/\alpha} \star \delta_1\}(u) du 
&= z^{\nu/\alpha} (zx^\alpha)^{-\nu/\alpha} e^{-zx^\alpha} =  x^{-\nu} e^{-zx^\alpha}
\end{align*}
Hence the validity of~(\ref{eq:power_stable_conv}) by uniqueness of Laplace transforms.
Explicitly, $\{\rho_{\nu/\alpha} \star \delta_1\}$ is
\begin{align*}
\{\rho_{\nu/\alpha} \star \delta_1\}(u) &= \int_0^u \rho_{\nu/\alpha} (u-v) \delta(v-1) dv  
= \rho_{\nu/\alpha}(u-1) \mathbbm{1}_{[1,\infty)} 
\end{align*}
which leads to~(\ref{eq:power_stable_conv1}).
Using  $\rho_{\nu/\alpha}(u)=u^{\nu/\alpha-1}/\Gamma(\nu/\alpha)$ and  $u\to1/u$ in~(\ref{eq:power_stable_conv1})
gives~(\ref{eq:power_stable_conv_betadensity}).
 \end{proof}

\section{Generalised Positive Linnik Law}
\label{sec:Linnik}

 \begin{proposition}
\label{prop:Linnik_rvproduct}
Consider two  independent positive random variables
\begin{enumerate}
\setlength{\itemsep}{0pt}
\item $S_{\alpha;z} \sim F_\alpha(z)$ $(0<\alpha<1)$,  where $F_\alpha(z)$ is  the  
$\alpha$-stable distribution, density $f_\alpha(\cdot\vert z)$ 
\item $G_{\gamma, \lambda}\sim G(\gamma>0, z>0)$, where  $G(\gamma, \lambda)$ is the   Gamma distribution. 
\end{enumerate}

Then the product $S_{\alpha;z}\, G_{\gamma, \lambda}^{1/\alpha}$ has a distribution  on  $(0,\infty)$ 
with density
\begin{align}
\ell_\alpha(x\vert \gamma, \lambda,z) &= \int_0^\infty f_\alpha(x\vert  zu) \, dG(u\vert \gamma, \lambda)
= \lambda^\gamma \int_0^\infty f_\alpha(x\vert  zu) \,  \rho_\gamma(u) \, e^{-\lambda u} \, du
\label{eq:Linnik_density}
\end{align}
The Laplace transform 
of~$(\ref{eq:Linnik_density})$ is 
\begin{align}
 \mathbb{E}\left[e^{-zs^\alpha G_{\gamma, \lambda}}\right]   
&\equiv \int_0^\infty e^{- zs^\alpha x} \,  dG(x\vert \gamma, \lambda)
= \left(\frac{\lambda}{\lambda+  zs^\alpha}\right)^\gamma \quad (s\ge 0)
\label{eq:Linnik_densityLT}
\end{align}
 \end{proposition}
\begin{proof} [Proof of Proposition~$\ref{prop:Linnik_rvproduct}$]
\label{proof:Linnik_rvproduct}
Set  $U\equiv G_{\gamma, \lambda}\sim F \equiv G(\gamma, \lambda)$ in Theorem~\ref{thm:gen_rvproduct}.
\end{proof}


\begin{remark}
\label{rem:Linnik}
Pakes~{\rm \cite[(4.4)]{Pakes95}} discussed  $S_{\alpha;z}\, G_{\gamma, 1}^{1/\alpha}$ with 
Laplace transform $(1+zs^\alpha)^{-\gamma}$, which he termed the  generalised positive  Linnik law,
after the  $\gamma=1$ symmetric case due to Linnik $($see Devroye~{\rm \cite{Devroye_Linnik}} for a concise discussion of the latter$)$.
\end{remark}

The rest of this section is based on  $S_{\alpha;1}\, G_{\gamma, \lambda}^{1/\alpha}$,  which induces
$\ell_\alpha(\cdot \vert \gamma, \lambda)\equiv\ell_\alpha(\cdot \vert \gamma, \lambda,z=1)$.
The following lemma,  involving gamma-Linnik convolution,  underpins subsequent discussion.

\begin{lemma}
\label{lem:gammaLinnik_conv}
Let  $\rho_\nu(x)=x^{\nu-1}/\Gamma(\nu)$ $(x,\nu>0)$ and $\rho_0(x)=\delta(x)$.
Then, for $\nu=\beta-\alpha\gamma\ge0$ 
 \begin{align}
\{\rho_{\beta-\alpha\gamma} \star \ell_\alpha(\cdot \vert \gamma, \lambda)\}(x) 
&=  \lambda^{\gamma} \, x^{\beta-1}E^\gamma_{\alpha,\beta}(-\lambda x^\alpha)
\label{eq:gammaLinnik_conv} 
\end{align}
The Laplace transform of the  density  
$\{\rho_{\beta-\alpha\gamma} \star f_\alpha(\cdot\vert  t)\}(x)\rho_\gamma(t)$, taken over $t$ for given $x$, is
\begin{align}
\int_0^\infty e^{-\lambda t} \, \{\rho_{\beta-\alpha\gamma} \star f_\alpha(\cdot\vert  t)\}(x) \, \rho_\gamma(t) \, dt 
&= x^{\beta-1}E^\gamma_{\alpha,\beta}(-\lambda x^\alpha)  \qquad (\lambda\ge0)
\label{eq:gammaLinnik_conv_LT_x}
 \end{align}
Hence the  Laplace transform of the probability density  $\Gamma(\beta)\{\rho_{\beta-\alpha\gamma} \star f_\alpha(\cdot\vert  t)\}(1)\rho_\gamma(t)$ is 
\begin{align}
\Gamma(\beta) \int_0^\infty e^{-\lambda t} \, \{\rho_{\beta-\alpha\gamma} \star f_\alpha(\cdot\vert  t)\}(1) \, \rho_\gamma(t) \, dt 
&= \Gamma(\beta)E^\gamma_{\alpha,\beta}(-\lambda)  \qquad (\lambda\ge0)
\label{eq:gammaLinnik_conv_LT}
 \end{align}
where $E^\gamma_{\alpha,\beta}(\cdot)$ is the Mittag-Leffler function~$(\ref{eq:ML3parseries})$, with $E^\gamma_{\alpha,\beta}(0)=1/\Gamma(\beta)$.
\end{lemma}
\begin{proof} [Proof of Lemma~$\ref{lem:gammaLinnik_conv}$]
\label{proof:gammaLinnik_conv}
We recall that the Laplace transform of $\rho_\nu(x)$ is $s^{-\nu}$ $(\nu\ge0, s>0)$.
Hence, by the convolution theorem, the Laplace transform of $\{\rho_{\beta-\alpha\gamma}\star \ell_\alpha(\cdot \vert \gamma, \lambda)\}(x)$ is
 \begin{align*}
 \int_0^\infty e^{-s x}  \{\rho_{\beta-\alpha\gamma}\star \ell_\alpha(\cdot \vert \gamma, \lambda)\}(x)\, dx
&=  \lambda^{\gamma} \, \frac{s^{\alpha\gamma-\beta}}{(\lambda+ s^\alpha)^{\gamma}} 
\end{align*}
By~(\ref{eq:ML3parseriesLT}), the right hand side is also the Laplace transform of 
$\lambda^{\gamma} \, x^{\beta-1}E^\gamma_{\alpha,\beta}(-\lambda x^\alpha)$.
Hence the validity of~(\ref{eq:gammaLinnik_conv})  by uniqueness of Laplace transforms.
Combining~(\ref{eq:Linnik_density}) and~(\ref{eq:gammaLinnik_conv}) gives~(\ref{eq:gammaLinnik_conv_LT_x})
 \begin{align*}
\int_0^\infty \{\rho_{\beta-\alpha\gamma} \star f_\alpha(\cdot\vert  t)\}(x)\rho_\gamma(t) \,  e^{-\lambda t} \, dt
&= \lambda^{-\gamma} \{\rho_{\beta-\alpha\gamma}\star \ell_\alpha(\cdot \vert \gamma, \lambda)\}(x)
=  x^{\beta-1}E^\gamma_{\alpha,\beta}(-\lambda x^\alpha)
\end{align*}
Setting $x=1$  and  dividing by $E^\gamma_{\alpha,\beta}(0)=1/\Gamma(\beta)$ to normalise
gives~(\ref{eq:gammaLinnik_conv_LT}), valid for $\lambda\ge0$.
\end{proof} 
\begin{remark}
\label{rem:gammaLinnik_conv}
Lemma~$\ref{lem:gammaLinnik_conv}$ rests on the idea  that 
$\{\rho_{\beta-\alpha\gamma} \star f_\alpha(\cdot\vert  t)\}(x)\rho_\gamma(t)$ is a  two-dimensional density in $(x,t)$, 
whose two-dimensional  Laplace transform $s^{\alpha\gamma-\beta}/(\lambda+ s^\alpha)^{\gamma}$  in $(\lambda,s)$ can be 
evaluated  as a Laplace transform in $s$ taken over  $x$, followed by a Laplace transform in $\lambda$ taken over $t$, or the other way round.
This argument  is inspired by Feller~{\rm \cite[XIII.8(b)]{Feller2}}.
\end{remark}


\section{Mittag-Leffler Laws from Linnik Laws}
\label{sec:ML_laws}

We derive a convolution-based representation of Mittag-Leffler densities.
\begin{proposition}
\label{prop:ML_laws}
The Mittag-Leffler distribution $P_{\alpha,\theta} \equiv{\rm ML}(\alpha,\theta)$  has the   density 
\begin{align}
p_{\alpha,\theta}(t) &= \frac{\Gamma(1+\theta)}{\Gamma(1+\theta/\alpha)} \,  t^{\theta/\alpha} \, p_\alpha(t) 
\qquad (0<\alpha<1,\, \theta>-\alpha)
 \label{eq:ML2parconv_density} \\
 p_\alpha(t) &=  \{\rho_{1-\alpha}\star f_\alpha(\cdot\vert t) \}(1) = \frac{1}{\alpha t} \, f_\alpha(1\vert t)
   \equiv   \frac{1}{\alpha} \,  f_\alpha(t^{-1/\alpha}) \, t^{-1/\alpha-1}
\label{ML1parconv_density} 
\end{align}
The Laplace transform of $p_{\alpha,\theta}(t)$ for $x\ge0$  is 
$\Gamma(1+\theta)E^{1+\theta/\alpha}_{\alpha,1+\theta}(-x)$ $(\theta>-\alpha)$. 
\end{proposition}

\begin{proof}[Proof of Proposition~$\ref{prop:ML_laws}$]
\label{proof:ML_laws}
Let $\beta=1+\theta$, $\gamma=1+\theta/\alpha>0$  
$\implies$  $\beta-\alpha\gamma=1-\alpha$.
Then~(\ref{eq:gammaLinnik_conv_LT}) becomes
\begin{align*}
\int_0^\infty e^{-x t} \, p_{\alpha,\theta}(t) \, dt 
&= \Gamma(1+\theta) E^{1+\theta/\alpha}_{\alpha,1+\theta}(-x)  \qquad (x\ge0)
\end{align*}
for the  density $p_{\alpha,\theta}(t)$ of the probability distribution ${\rm ML}(\alpha,\theta)$
\begin{align*}
p_{\alpha,\theta}(t) 
&= \Gamma(1+\theta) \{\rho_{1-\alpha} \star f_\alpha(\cdot\vert  t)\}(1) \,  \rho_{1+\theta/\alpha}(t) \\
&\equiv \frac{\Gamma(1+\theta)}{\Gamma(1+\theta/\alpha)}   \,  t^{\theta/\alpha} \, p_\alpha(t) \qquad (\theta>-\alpha)
\end{align*}
$p_\alpha(t) =  \{\rho_{1-\alpha}\star f_\alpha(\cdot\vert t) \}(1)$.  
We   also  recall  the  identity $\{\rho_{1-\alpha}\star f_\alpha(\cdot\vert  t)\}(1)\equiv f_\alpha(1\vert t)/\alpha t$  
of~(\ref{eq:stableLTderivdensity}) 
and $f_\alpha(1\vert t) \equiv f_\alpha(t^{-1/\alpha}) t^{-1/\alpha}$.
\end{proof}

Another representation that readily follows from convolution is the following 
\begin{proposition}
\label{prop:ML_beta}
The densities  of ${\rm ML}(\alpha,\theta)$  and ${\rm ML}(\alpha,\theta+1)$   are related by
\begin{align}
p_{\alpha,\theta}(t) 
&=  \int_0^1 p_{\alpha,\theta+1}(t/u) \, {\rm beta}\left(u \left.\right\vert \tfrac{\theta}{\alpha}+1,\tfrac{1}{\alpha}-1\right) \,\frac{du}{u} 
\label{eq:ML2parconvint1}
\end{align}
Let  $M_{\alpha,\theta} \sim {\rm ML}(\alpha,\theta)$ and $B_{\theta/\alpha+1,1/\alpha-1} \sim {\rm Beta}(\theta/\alpha+1,1/\alpha-1)$.
Then~$(\ref{eq:ML2parconvint1})$ implies
\begin{align}
M_{\alpha,\theta} &\overset{d}{=} M_{\alpha,\theta+1} \times  B_{\theta/\alpha+1,1/\alpha-1}
\label{eq:ML_beta_product}
\end{align}
where $M_{\alpha,\theta+1}$ and  $B_{\theta/\alpha+1,1/\alpha-1}$ are independent.
\end{proposition}

\begin{proof}[Proof of Proposition~$\ref{prop:ML_beta}$]
\label{proof:ML_beta}
Theorem~\ref{thm:rho_stable_convolution} for  $\{\rho_{1-\alpha}\star f_\alpha(\cdot\vert t) \}(1)$ leads to
\begin{align*} 
p_{\alpha,\theta}(t) 
&= \frac{\Gamma(1+\theta)}{\Gamma(1+\theta/\alpha)} \,  t^{\theta/\alpha} \, \{\rho_{1-\alpha}\star f_\alpha(\cdot\vert t) \}(1) \\
&=  \frac{\Gamma(1+\theta)}{\Gamma(1+\theta/\alpha)\Gamma(1/\alpha-1)} 
 \int_0^1 f_\alpha(1 \vert t/u) \, (t/u)^{(\theta+1)/\alpha-1} \, u^{\theta/\alpha} (1-u)^{1/\alpha-2}  \, \frac{du}{u} 
\\
&=  \int_0^1 p_{\alpha,\theta+1}(t/u) \, {\rm beta}\left(u \left.\right\vert \tfrac{\theta}{\alpha}+1,\tfrac{1}{\alpha}-1\right) \,\frac{du}{u} 
\end{align*}
By~(\ref{eq:productdensity}), this integral  is the density of a product 
of  independent random variables $M_{\alpha,\theta+1}$ and  $B_{\theta/\alpha+1,1/\alpha-1}$, 
{\it i.e.}\ $M_{\alpha,\theta} \overset{d}{=}  M_{\alpha,\theta+1}  B_{\theta/\alpha+1,1/\alpha-1}$
for random variables as defined above.
\end{proof}
The identity~(\ref{eq:ML_beta_product}) is well-known, the approach here rests on the convolution $\{\rho_{1-\alpha}\star f_\alpha(\cdot\vert t) \}(1)$. 

\subsection{Mittag-Leffler Markov Chain}
\label{sec:MLMC}

The Mittag-Leffler Markov chain  underpins the study of stable trees and graphs 
(Goldschmidt~{\it  et al.}~\cite{Goldschmidt22}, Goldschmidt and Haas~\cite{GoldschmidtHaas},  Haas~{\it  et al.}~\cite{Haas}, James~\cite{James2015}, Rembart and Winkel~\cite{RembartWinkel2018, RembartWinkel2023}, S\'{e}nizergues~\cite{Senizergues}). 
The following lemma states a  well-known result, but derivations  in the literature tend to be more elaborate
(however, see James~\cite[Proposition~3.1]{James2015} for a compact proof,  although the overall 
 discussion  rests on ${\rm PD}(\alpha,\theta)$).
\begin{lemma}
\label{lem:MLMC}
For  $M_{\alpha,\theta}\sim{\rm ML}(\alpha,\theta)$, $\{M_{\alpha,\theta+k}, k\ge0\}$ is a time-homogeneous  Markov chain 
with the  $k\to k+1$ transition probability given by the conditional density
\begin{align}
\Pr(M_{\alpha,\theta+k+1}\equiv u \vert  M_{\alpha,\theta+k}\equiv t)
&=  \frac{\alpha u}{\Gamma\left(\tfrac{1}{\alpha}-1\right)} \frac{p_\alpha(u)}{p_\alpha(t)} (u-t)^{1/\alpha-2} \quad (u>t)
\label{eq:MLMC}
\end{align}
\end{lemma}

\begin{proof}[Proof of Lemma~$\ref{lem:MLMC}$]
\label{proof:MLMC}
Set $\Pr(M_{\alpha,\theta+k} \equiv t) = p_{\alpha,\theta+k}(t)$ and $\Pr(M_{\alpha,\theta+k+1} \equiv u)= p_{\alpha,\theta+k+1}(u)$.
By~(\ref{eq:jointdensity}), the joint density  $\Pr(t,u)=\Pr(t\vert u)\Pr(u)= \Pr(u \vert t)  \Pr(t)$ is 
\begin{align*} 
\Pr(t,u) &=  \Pr(t\vert u)\times\Pr(u) 
  = \frac{1}{u} \, {\rm beta}\left(\tfrac{t}{u} \left.\right\vert \tfrac{\theta+k}{\alpha}+1, \tfrac{1}{\alpha}-1\right)\times
p_{\alpha,\theta+k+1}(u)
\end{align*}
Hence the $k\to k+1$ transition probability  $\Pr(M_{\alpha,\theta+k+1}\equiv u \vert  M_{\alpha,\theta+k}\equiv t)$ is
\begin{align*} 
\Pr(u \vert t)  
= \frac{\Pr(t, u)}{\Pr(t)}
&= \frac{p_{\alpha,\theta+k+1}(u)}{p_{\alpha,\theta+k}(t)}  \frac{1}{u} \,
{\rm beta}\left(\tfrac{t}{u} \left.\right\vert \tfrac{\theta+k}{\alpha}+1, \tfrac{1}{\alpha}-1\right) \\
&= \frac{\alpha  u^{1/\alpha-1}}{\Gamma\left(\tfrac{1}{\alpha}-1\right)} \frac{p_\alpha(u)}{p_\alpha(t)} \left(1-\frac{t}{u}\right)^{1/\alpha-2} 
\quad (u>t)
\end{align*}
which is independent of $k$ (and $\theta$), {\it  i.e.}\ $\{M_{\alpha,\theta+k}, k\ge0\}$ is a time-homogeneous Markov chain.
\end{proof}

\section{A New Family of Laws}
\label{sec:new_family}
 
 
 \begin{proposition}
\label{prop:ML_rvproduct}
Consider two  independent  positive random variables
\begin{itemize}
\setlength{\itemsep}{0pt}
\item $S_{\sigma;z} \sim F_\sigma(z>0)$ $(0<\sigma<1)$,  where $F_\sigma(z)$ is  the  
$\sigma$-stable distribution, density $f_\sigma(\cdot\vert z)$ 
\item $M_{\alpha,\theta} \sim P_{\alpha,\theta} \equiv {\rm ML}(\alpha,\theta)$ $(0<\alpha<1, \theta>-\alpha)$,  
the   Mittag-Leffler distribution 
\end{itemize}
Then the following holds
\begin{enumerate}
\setlength{\itemsep}{0pt}
\item The product $T_{\alpha,\theta,\sigma,z}=S_{\sigma;z} \, M_{\alpha,\theta}^{1/\sigma}$  has a distribution $H_{\alpha,\theta; \sigma}(z)$ with mixture density
\begin{align}
h_{\alpha,\theta; \sigma}(t \vert z)
&=  \int_0^\infty  f_\sigma(t \vert z u) \, dP_{\alpha,\theta}(u)
\label{eq:ML_rvproduct_density}
\end{align}
\item The  Laplace transform 
of~$(\ref{eq:ML_rvproduct_density})$ is 
\begin{align}
 \mathbb{E}\left[e^{-xT_{\alpha,\theta,\sigma,z}}\right]  &=  \mathbb{E}\left[e^{-zx^\sigma M_{\alpha,\theta}}\right] 
 = \Gamma(1+\theta)E^{1+\theta/\alpha}_{\alpha,1+\theta}(-z x^\sigma)   
\label{eq:ML_densityLT} 
\end{align}
\item $h_{\alpha,\theta; \sigma}(t \vert z)$  may  be expressed as the infinite series
\begin{align}
h_{\alpha,\theta; \sigma}(t \vert z)
&= -  \frac{1}{\pi}\,   \sum_{k=1}^\infty  \frac{(-z)^k}{k!} \sin(\pi\sigma k) \; \frac{\Gamma(\sigma k+1)}{t^{\sigma k+1}}  \; \mu_{\alpha,\theta;k}
\label{eq:ML_rvproduct_density_iinfseries}
\end{align}
where $\mu_{\alpha,\theta;k}$ is the $k^{\rm th}$ moment~$(\ref{eq:ML2parMoment})$ of the Mittag-Leffler distribution ${\rm ML}(\alpha,\theta)$.
\end{enumerate}
\end{proposition}

\begin{proof} [Proof of Proposition~$\ref{prop:ML_rvproduct}$]
\label{proof:ML_rvproduct}
Set  $U \equiv M_{\alpha,\theta} \sim F\equiv {\rm ML}(\alpha,\theta)$ in Theorem~\ref{thm:gen_rvproduct} 
to prove~(\ref{eq:ML_rvproduct_density}) and~(\ref{eq:ML_densityLT}).
Since ${\rm ML}(\alpha,\theta)$ has moments $\{\mu_{\alpha,\theta;k}\}$ of all order,  
 Corollary~\ref{cor:hPollardinfsum} gives~(\ref{eq:ML_rvproduct_density_iinfseries}).
\end{proof}

 
Particular instances of the densities $h_{\alpha,\theta; \sigma}(t \vert z)$    have been comprehensively  studied 
under various guises in the probabilistic literature, as discussed next.

\subsection{Lamperti-type Laws}
\label{sec:Lamperti-type}

James~\cite{James_Lamperti}  studied  the variables $X_{\alpha,\theta} \overset{d}{=} S_{\alpha}/S_{\alpha,\theta}$
where $S_{\alpha,\theta}$ has a distribution with density
\begin{align*}
 \frac{\Gamma(1+\theta)}{\Gamma(1+\theta/\alpha)} \,  t^{-\theta} f_\alpha(t) \qquad (\theta>-\alpha)
 \end{align*}
and $S_\alpha \equiv S_{\alpha,\theta=0}$ is the $\alpha$-stable variable ($S_{\alpha;z=1}$ in the notation of this paper).
It can readily be shown that  $S_{\alpha,\theta}  \overset{d}{=} M_{\alpha,\theta}^{-1/\alpha}$ so that,
if $T_{\alpha,\theta,\alpha,z} \sim H_{\alpha,\theta; \alpha}(z)$
\begin{align}
 X_{\alpha,\theta}  &\overset{d}{=} \frac{S_{\alpha}}{S_{\alpha,\theta}} \overset{d}{=} S_{\alpha;1} M_{\alpha,\theta}^{1/\alpha}
\overset{d}{=} T_{\alpha,\theta,\alpha,1} 
\label{eq:Lamperti_type_rv}
\end{align}
Hence the $X_{\alpha,\theta}$ variables, described by James  as  having  ``Lamperti-type laws'', arise in our context as the $\sigma=\alpha$ instance of 
Proposition~\ref{prop:ML_rvproduct} with explicit densities $h_{\alpha,\theta; \alpha}(\cdot \vert 1)$.
For general $z>0$, we can define $X_{\alpha,\theta;z} \overset{d}{=} S_{\alpha;z} M_{\alpha,\theta}^{1/\alpha}
\overset{d}{=} T_{\alpha,\theta,\alpha,z}$ 
with density $h_{\alpha,\theta; \alpha}(\cdot \vert z)$.

The Lamperti laws defined by $\{\sigma=\alpha, \theta=0\}$ are known to have simple densities.
Explicitly
 \begin{alignat}{3}
 h_{\alpha,0; \alpha}(t \vert z)
 &=  \int_0^\infty  f_\alpha(t \vert z u) \, dP_{\alpha}(u)    \quad && [\textrm{by } (\ref{eq:ML_rvproduct_density})] \nonumber  \\
 &=  \frac{1}{\pi t}\,  {\rm Im} \sum_{k=0}^\infty  (-z e^{-i\pi\alpha} t^{-\alpha})^k  && [\textrm{by } (\ref{eq:ML_rvproduct_density_iinfseries});\,
                                       \Gamma(\alpha k+1)\mu_{\alpha,0;k} = k!] \nonumber  \\
 &= \frac{1}{\pi t}\,  {\rm Im} \frac{1}{1+z e^{-i\pi\alpha} t^{-\alpha}} \nonumber  \\
 &=  \frac{\sin\pi\alpha}{\pi}\, \frac{z\, t^{\alpha-1}} {z^2+2z \, t^\alpha \cos\pi\alpha+t^{2\alpha}} 
\label{eq:hPollard5} 
\end{alignat}
whose Laplace transform is  $E_\alpha(-z x^\alpha)$, as well-known.
The change of variable $t=u/(1-u)$ or $u=t/(1+t)$ with $dt=du/(1-u)^2$  transforms
 $h_{\alpha,0; \alpha}(t \vert z) dt$ to $ g_{\alpha,0; \alpha}(u \vert z) du$, where
 \begin{align}
 g_{\alpha,0; \alpha}(u \vert z)
 &=  \frac{\sin\pi\alpha}{\pi}\, \frac{z\, u^{\alpha-1}(1-u)^{\alpha-1}} {z^2(1-u)^{2\alpha}+2z \, u^\alpha (1-u)^\alpha \cos\pi\alpha+u^{2\alpha}} 
 \quad (0<u<1)
\label{eq:hPollard6} 
\end{align}
An expression analogous to  the definition $p_\alpha(t)=f_\alpha(1\vert t)/\alpha t$ of the Mittag-Leffler density  is 
 \begin{align}
\frac{1}{\alpha z} h_{\alpha,0; \alpha}(1 \vert z)
 &=  \frac{\sin\pi\alpha}{\pi\alpha}\, \frac{1} {z^2+2z \cos\pi\alpha+1} 
\label{eq:hPollard7} 
\end{align}

The distributions  $H_{\alpha,0; \alpha}(z)$  are  the well-known generalised arcsine laws 
 arising  in the study of  occupation times for  Markov processes
(Bertoin~{\it et al.}~\cite{Bertoin}, Feller~\cite[XIV.3]{Feller2}, James~\cite{James_Bernoulli, James_Lamperti}, 
James~{\it et al.}~\cite{JamesRoynetteYor},
Lamperti~\cite{Lamperti},  Pitman~\cite{Pitman_CSP, Pitman2018}).
The   Mittag-Leffler distribution  ${\rm ML}(\alpha)$   arises as 
the limiting distribution  for the occupation time of a single-state set (Darling and Kac~{\cite{DarlingKac}). 
The distributions are often studied from the perspective of a cascade of products and ratios of stable random variables.
Explicitly, define the   following random variables:  
\begin{enumerate}
\setlength{\itemsep}{0pt}
\item $S_{\alpha;z}$ with stable  density $f_\alpha(\cdot \vert z)$ 
\item $M_\alpha$ with  Mittag-Leffler distribution ${\rm ML}(\alpha)$ 
\item $X_{\alpha;z}$ with  density $h_{\alpha,0; \alpha}(\cdot\vert z)$  in the form shown in~(\ref{eq:hPollard5}) 
\item $U_{\alpha;z}$ with  density $g_{\alpha,0; \alpha}(\cdot \vert z)$  in the form shown in~(\ref{eq:hPollard6}) 
\item $Z_\alpha$ with  density $h_{\alpha,0; \alpha}(1 \vert z)/\alpha z$ in the form shown in~(\ref{eq:hPollard7})
\end{enumerate}
Then the following equalities in distribution are  known:
\begin{alignat*}{2}
M_\alpha &\overset{d}{=} S_{\alpha;1}^{-\alpha}   & & (\textrm{Pitman~\cite{Pitman_CSP},  James~\cite{James_Lamperti}})  \\ 
X_{\alpha;z} &\overset{d}{=}  \frac{S_{\alpha;z}}{S_{\alpha;1}} & &  (\textrm{$z=1$: James~\cite{James_Bernoulli, James_Lamperti}, Pitman~\cite{Pitman2018}}) \\ 
U_{\alpha;z} &\overset{d}{=}  \frac{X_{\alpha;z}}{1+X_{\alpha;z}} \quad & & 
(\textrm{James~{\it et al.}~\cite{JamesRoynetteYor}, Pitman~\cite{Pitman2018}}) \\
Z_\alpha &\overset{d}{=}  X_{\alpha;1}^\alpha  & & (\textrm{Chaumont and Yor~\cite[4.21.3(p116)]{ChaumontYor}})
\end{alignat*}
We may also write 
$X_{\alpha;z} \overset{d}{=}    S_{\alpha;z}M_\alpha^{1/\alpha} \overset{d}{=} T_{\alpha,0,\alpha,z}  \sim H_{\alpha,0; \alpha}(z)$.

\begin{remark}
Bondesson~{\rm \cite[3.2.4(p38)]{Bondesson}} showed that   $S_{\alpha;1}\, G_{\gamma, 1}^{1/\alpha}$  has
a Thorin measure $($in the language of infinitely divisible distributions known as generalised gamma convolutions$)$
with density
\begin{align*}
 \tau_\alpha(t\vert \gamma) 
 = \gamma\alpha \; \frac{\sin\pi\alpha}{\pi} \, \frac{t^{\alpha-1}}{1+2  t^\alpha\cos\pi\alpha+t^{2\alpha}}
\end{align*}
Bondesson used the  theory of Pick functions,  amounting to  Stieltjes inversion of $($minus$)$ the 
logarithm derivative of the Linnik Laplace transform 
$\widetilde \ell_\alpha(s\vert \gamma) \equiv \widetilde \ell_\alpha(s\vert  \gamma,1,1)$ 
\begin{align*}
-\frac{d}{ds}\log \widetilde \ell_\alpha(s\vert \gamma)
&= \gamma  \alpha \, \frac{ s^{\alpha-1}}{1+  s^\alpha} \qquad (s>0)
\end{align*}
\end{remark}

Finally, we turn to a natural generalisation of the Mittag-Leffler distributions  ${\rm ML}(\alpha,\theta)$. 
 
\section{Mittag-Leffler Generalisation}
\label{sec:generalisation}
Returning to Lemma~\ref{lem:gammaLinnik_conv}, let $\beta\to\beta+\theta$ and $\gamma\to\gamma+\theta/\alpha>0$.
With $\beta-\alpha\gamma\ge0$ remaining invariant, this generalises~(\ref{eq:gammaLinnik_conv_LT}) 
so  that for   $-\theta<\alpha\gamma\le\beta$
\begin{align}
\int_0^\infty  e^{-xt} \, p_{\alpha,\theta}(t\vert \beta,\gamma) \, dt   
&= \Gamma(\beta+\theta)E^{\gamma+\theta/\alpha}_{\alpha,\beta+\theta}(-x) \qquad (x\ge0)
\\
\textrm{where} \quad
p_{\alpha,\theta}(t\vert \beta,\gamma) 
&= \frac{\Gamma(\beta+\theta)}{\Gamma(\gamma+\theta/\alpha)} \, t^{\gamma+\theta/\alpha-1}
 \{\rho_{\beta-\alpha\gamma}\star f_\alpha(\cdot\vert  t)\}(1) 
\label{eq:rho_Linnik_convdensitygen}
\end{align}
$p_{\alpha,\theta}(\cdot \vert \beta,\gamma)$ is the density of a more general distribution 
we denote by 
$P_{\alpha,\theta\vert\beta,\gamma}\equiv {\rm ML}(\alpha,\theta \vert \beta,\gamma)$,
with $\beta=\gamma=1$ recovering the Mittag-Leffler distribution 
${\rm ML}(\alpha,\theta)\equiv{\rm ML}(\alpha,\theta \vert 1,1)$.
The moments of ${\rm ML}(\alpha,\theta \vert \beta,\gamma)$ take  a  more general form of the moments~$(\ref{eq:ML2parMoment})$ of 
${\rm ML}(\alpha,\theta)$
\begin{align}
\mu_{\alpha,\theta\vert\beta,\gamma;k} \equiv \int_0^\infty t^k \, P_{\alpha,\theta}(t\vert \beta,\gamma) 
&= \frac{\Gamma(\beta+\theta)\Gamma(\gamma+\theta/\alpha+k)}{\Gamma(\gamma+\theta/\alpha)\Gamma(\beta+\theta+k\alpha)}  
\quad  (-\theta<\alpha\gamma<\beta) 
 \label{eq:rho_Linnik_convdensitygen_moments} 
 \end{align}

When $\beta-\alpha\gamma=0$, the convolution term in~(\ref{eq:rho_Linnik_convdensitygen}) is
$\{\rho_0\star f_\alpha(\cdot\vert  t)\}(1)= f_\alpha(1\vert  t)$ since  $\rho_0(t)=\delta(t)$.
Hence, with $\beta=\alpha\gamma$ and  recalling that $f_\alpha(1\vert  t) = \alpha t \, p_\alpha(1\vert  t)$
\begin{alignat}{5}
p_{\alpha,\theta}\left(t \left. \right\vert \beta,\tfrac{\beta}{\alpha}\right) 
&= \frac{\Gamma(1+\beta+\theta)}{\Gamma\left(1+(\beta+\theta)/\alpha\right)} \, 
t^{(\beta+\theta)/\alpha} \, p_\alpha(1\vert  t)  &&\equiv p_{\alpha,\beta+\theta}(t) \quad &&(\beta+\theta>-\alpha) 
\label{eq:ML4par_identity}
\end{alignat}
{\it i.e.}\ $P_{\alpha,\theta\vert\beta,\beta/\alpha} = P_{\alpha,\beta+\theta}$ $(\beta+\theta>-\alpha)$.
Equivalently, $P_{\alpha,\theta \vert\alpha\gamma,\gamma} = P_{\alpha,\alpha\gamma+\theta}$  $(\alpha\gamma+\theta>-\alpha)$.

A correspondingly more general variant of Proposition~\ref{prop:ML_beta} takes the following form
\begin{proposition}
\label{prop:rho_Linnik_convdensitygen_beta}
The densities of ${\rm ML}(\alpha,\theta \left.\right\vert \beta,\gamma)$  and 
${\rm ML}(\alpha,\beta+\theta) \equiv {\rm ML}(\alpha,\theta \left.\right\vert \beta,\beta/\alpha)$ 
for $(-\theta<\alpha\gamma<\beta)$ are related by
\begin{align}
p_{\alpha,\theta}(t\vert \beta, \gamma)
 &= \int_0^1 p_{\alpha,\beta+\theta}\left(\tfrac{t}{u}\right) \, 
 {\rm beta} \left(u \left.\right\vert \tfrac{\theta}{\alpha}+\gamma, \tfrac{\beta}{\alpha}-\gamma\right) \frac{du}{u}
 \label{eq:rho_Linnik_convdensitygen_beta} \\
 &\equiv \int_0^1 p_{\alpha,\theta}\left(\tfrac{t}{u} \left. \right\vert \beta,\tfrac{\beta}{\alpha}\right) \, 
 {\rm beta} \left(u \left.\right\vert \tfrac{\theta}{\alpha}+\gamma, \tfrac{\beta}{\alpha}-\gamma\right) \frac{du}{u}
 \label{eq:rho_Linnik_convdensitygen_beta_1}
  \end{align}
Equivalently, define   the random variables 
\begin{align*}
M_{\alpha,\beta+\theta} &\sim {\rm ML}(\alpha,\beta+\theta) \\
B_{\theta/\alpha+\gamma,\beta/\alpha-\gamma} &\sim {\rm Beta}(\theta/\alpha+\gamma,\beta/\alpha-\gamma) \\
M_{\alpha,\theta\vert\beta,\gamma} &\sim {\rm ML}(\alpha,\theta \left.\right\vert \beta,\gamma)
\end{align*}
where the first two are independent. 
Then 
 \begin{align}
 M_{\alpha,\theta\vert\beta,\gamma} &\overset{d}{=} M_{\alpha,\beta+\theta} \, B_{\theta/\alpha+\gamma,\beta/\alpha-\gamma}
  \label{eq:ML4par_rvproduct} \\
  \textrm{or}\quad
 M_{\alpha,\theta\vert\beta,\gamma} &\overset{d}{=} 
 M_{\alpha,\theta\vert\beta,\beta/\alpha} \, B_{\theta/\alpha+\gamma,\beta/\alpha-\gamma}
  \label{eq:ML4par_rvproduct1}
 \end{align}
  \end{proposition}
 This generates  the  time-homogeneous Mittag-Leffler Markov chain of Section~\ref{sec:MLMC} for $\beta=\gamma=1$.

 \begin{proof}[Proof of Proposition~$\ref{prop:rho_Linnik_convdensitygen_beta}$]
\label{proof:rho_Linnik_convdensitygen_beta}
By  Theorem~\ref{thm:rho_stable_convolution}, 
$\{\rho_{\beta-\alpha\gamma}\star f_\alpha(\cdot\vert  t)\}(1)$ may be represented as 
\begin{align*}
&\{\rho_{\beta-\alpha\gamma} \star f_\alpha(\cdot\vert t)\}(1)
=  \frac{t^{\beta/\alpha-\gamma}}{\Gamma(\tfrac{\beta}{\alpha}-\gamma)}
  \int_0^1 f_\alpha(1\vert t/u) \, u^{\gamma-\beta/\alpha}  (1-u)^{\beta/\alpha-\gamma-1}  \, \frac{du}{u} \\
\implies \;
 &p_{\alpha,\theta}(t \vert \beta,\gamma) 
= \frac{\Gamma(\beta+\theta)}{\Gamma(\gamma+\theta/\alpha)} \, t^{\gamma+\theta/\alpha-1} \,
\{\rho_{\beta-\alpha\gamma} \star f_\alpha(\cdot\vert t)\}(1)  \\
&=  \frac{\Gamma(\beta+\theta)}{\Gamma(\gamma+\theta/\alpha)\Gamma(\tfrac{\beta}{\alpha}-\gamma)}
  \int_0^1 f_\alpha(1\vert \tfrac{t}{u}) \, (\tfrac{t}{u})^{(\beta+\theta)/\alpha-1} 
    u^{\theta/\alpha+\gamma-1} (1-u)^{\beta/\alpha-\gamma-1}  \, \frac{du}{u}  \\
&= \frac{\Gamma(\beta+\theta)}{\Gamma(\gamma+\theta/\alpha)} \frac{1}{\Gamma(\beta/\alpha-\gamma)}
    \frac{\alpha\Gamma(1+(\beta+\theta)/\alpha)}{\Gamma(1+\beta+\theta)} 
    \frac{\Gamma(\gamma+\theta/\alpha)\Gamma(\beta/\alpha-\gamma)}{\Gamma((\beta+\theta)/\alpha)} \\ 
&\qquad  \times  \int_0^1 p_{\alpha,\beta+\theta}\left(\tfrac{t}{u}\right) \, 
 {\rm beta} \left(u \left.\right\vert \tfrac{\theta}{\alpha}+\gamma, \tfrac{\beta}{\alpha}-\gamma\right) \frac{du}{u} \\
&= \int_0^1 p_{\alpha,\beta+\theta}\left(\tfrac{t}{u}\right) \, 
 {\rm beta} \left(u \left.\right\vert \tfrac{\theta}{\alpha}+\gamma, \tfrac{\beta}{\alpha}-\gamma\right) \frac{du}{u}
 \end{align*}
 with the  product of gamma functions evaluating to 1, thereby proving~(\ref{eq:rho_Linnik_convdensitygen_beta}). 
 The identity~(\ref{eq:ML4par_identity}) gives~(\ref{eq:rho_Linnik_convdensitygen_beta_1}).
 By~(\ref{eq:productdensity}), (\ref{eq:rho_Linnik_convdensitygen_beta}) is the distribution of 
 $M_{\alpha,\beta+\theta} \, B_{\theta/\alpha+\gamma,\beta/\alpha-\gamma}$, 
 hence~(\ref{eq:ML4par_rvproduct}) and~(\ref{eq:ML4par_rvproduct1}).
\end{proof}
\begin{remark}
\label{rem:ML4par_rvproduct}
For $\gamma=0$, $(\ref{eq:ML4par_rvproduct})$  reduces to~Banderier~{\it  et al.}~{\rm \cite[Lemma~3.14]{Banderier}}.
For  $\theta=0$, $(\ref{eq:ML4par_rvproduct1})$  
reduces to Bertoin and Yor~{\rm \cite[Lemma~6(ii)]{BertoinYor}}, 
where $M_{\alpha,0 \vert\beta,\gamma} \equiv J^{(\alpha)}_{\beta,\gamma}$ in their notation.
Also see Ho~{\it  et al.}~{\rm \cite[Proposition~2.1, Remark~2.3]{HoJamesLau}}
and M\"{o}hle~{\rm \cite[Corollary~3]{Mohle}}.
\end{remark}

\begin{proposition}
\label{prop:ML4par_rvproduct}
Consider two  independent  positive random variables
\begin{itemize}
\setlength{\itemsep}{0pt}
\item $S_{\sigma;z} \sim F_\sigma(z>0)$ $(0<\sigma<1)$,  where $F_\sigma(z)$ is  the  
$\sigma$-stable distribution, density $f_\sigma(\cdot\vert z)$ 
\item $M_{\alpha,\theta \vert\beta,\gamma} \sim P_{\alpha,\theta\vert\beta,\gamma}\equiv {\rm ML}(\alpha,\theta \vert \beta,\gamma)$ 
$(0<\alpha<1,  -\theta<\alpha\gamma\le\beta)$
\end{itemize}
Then the following holds
\begin{enumerate}
\setlength{\itemsep}{0pt}
\item The product $S_{\sigma;z} \, M_{\alpha,\theta\vert\beta,\gamma}^{1/\sigma}$  has a distribution 
$H_{\alpha,\theta; \sigma}(z,\beta,\gamma)$ with mixture density:
\begin{align}
h_{\alpha,\theta; \sigma}(t \vert z , \beta,\gamma)
&=  \int_0^\infty  f_\sigma(t \vert z u) \, dP_{\alpha,\theta}(u\vert\beta,\gamma)
\label{eq:ML4par_rvproduct_density}
\end{align}
\item The  Laplace transform 
of~$(\ref{eq:ML4par_rvproduct_density})$ is 
\begin{align}
 \mathbb{E}\left[e^{-xS_{\sigma;z} M_{\alpha,\theta \vert\beta,\gamma}^{1/\sigma}}\right]  
 &=  \mathbb{E}\left[e^{-zx^\sigma M_{\alpha,\theta \vert\beta,\gamma}}\right] 
 = \Gamma(\beta+\theta)E^{\gamma+\theta/\alpha}_{\alpha,\beta+\theta}(-z x^\sigma) 
\label{eq:ML4par_rvproduct_LT}
\end{align}
\item $h_{\alpha,\theta; \sigma}(t \vert z , \beta,\gamma)$  may  be expressed as the infinite series
\begin{align}
h_{\alpha,\theta; \sigma}(t \vert z , \beta,\gamma)
&= -  \frac{1}{\pi}\,   \sum_{k=1}^\infty  \frac{(-z)^k}{k!} \sin(\pi\sigma k) \; \frac{\Gamma(\sigma k+1)}{t^{\sigma k+1}}  \;
 \mu_{\alpha,\theta\vert\beta,\gamma;k} 
\label{eq:ML4par_rvproduct_density_infseries}
\end{align}
where $\mu_{\alpha,\theta\vert\beta,\gamma;k}$ is the $k^{\rm th}$ moment~$(\ref{eq:rho_Linnik_convdensitygen_moments})$ 
of the distribution ${\rm ML}(\alpha,\theta \vert \beta,\gamma)$.
\end{enumerate}
\end{proposition}

\begin{proof} [Proof of Proposition~$\ref{prop:ML4par_rvproduct}$]
\label{proof:ML4par_rvproduct}
$U \equiv M_{\alpha,\theta \vert\beta,\gamma} \sim F \equiv {\rm ML}(\alpha,\theta \vert \beta,\gamma)$  in Theorem~\ref{thm:gen_rvproduct}
to prove~(\ref{eq:ML4par_rvproduct_density}) and~(\ref{eq:ML4par_rvproduct_LT}).
Since ${\rm ML}(\alpha,\theta \vert \beta,\gamma)$ has moments $\{\mu_{\alpha,\theta\vert\beta,\gamma;k}\}$ of all order,  
 Corollary~\ref{cor:hPollardinfsum} gives~(\ref{eq:ML4par_rvproduct_density_infseries}).
 \end{proof}

\begin{remark}
\label{rem:ML4par_rvproduct1}
Proposition~$\ref{prop:ML4par_rvproduct}$ subsumes  Proposition~$\ref{prop:ML_rvproduct}$ as a special case,
but the presentation order  in this paper better captures the  thought process  and arguably makes for easier understanding.
\end{remark}

\subsection{More on Lamperti-type Laws}
\label{sec:Lamperti-type_gen}

In Section~\ref{sec:Lamperti-type}, we discussed Lamperti-type densities with $\sigma=\alpha$; {\it i.e.}\
$h_{\alpha,\theta; \alpha}(t \vert z)$ induced by the  random variable product $S_{\alpha;z} \, M_{\alpha,\theta}^{1/\alpha}$.
We conclude by exploring   the more general case of densities   $h_{\alpha,\theta; \alpha}(t \vert z , \beta,\gamma)$ induced  by 
$S_{\alpha;z} \, M_{\alpha,\theta\vert\beta,\gamma}^{1/\alpha}$,
{\it i.e.}\ the $\sigma=\alpha$ case of~(\ref{eq:ML4par_rvproduct_density})    
with Laplace transform $\Gamma(\beta+\theta)E^{\gamma+\theta/\alpha}_{\alpha,\beta+\theta}(-z x^\alpha)$. 
We focus on $\beta+\theta=1$ as a particular instance that simplifies the infinite sum representation~(\ref{eq:ML4par_rvproduct_density_infseries}).

\begin{proposition}
\label{prop:Lamperti-type_gen}
In Proposition~$\ref{prop:ML4par_rvproduct}$, let  $\sigma=\alpha$ and  $\beta+\theta=1 \implies -\theta<\alpha\gamma \le1-\theta$
or $0< \alpha\gamma+\theta\le 1$. 
Then the product $S_{\alpha;z} \, M_{\alpha,\theta\vert 1-\theta,\gamma}^{1/\alpha}$  has a distribution with density
\begin{align}
h_{\alpha,\theta; \alpha}(t \vert z , 1-\theta,\gamma)
&=  \int_0^\infty  f_\alpha(t \vert z u) \, dP_{\alpha,\theta}(u\vert 1-\theta,\gamma)
\label{eq:Lamperti-type_gen_density} \\
&= \frac{1}{\pi} \, {\rm Im} \,    \frac{t^{\alpha\gamma+\theta-1}}{(z e^{-i\pi\alpha}+t^{\alpha})^{\gamma+\theta/\alpha}}
\label{eq:Lamperti-type_gen_density_infseries}
\end{align}
with Laplace transform $E^{\gamma+\theta/\alpha}_{\alpha,1}(-z x^\alpha)$.
For $\alpha\gamma+\theta=\alpha$,~$(\ref{eq:Lamperti-type_gen_density_infseries})$  reduces to the density~$(\ref{eq:hPollard5})$ of  the Lamperti laws. 
\end{proposition}

\begin{proof} [Proof of Proposition~$\ref{prop:Lamperti-type_gen}$]
For  $\sigma=\alpha$ and $\beta+\theta=1$, 
the term $\Gamma(\sigma k+1) \, \mu_{\alpha,\theta\vert\beta,\gamma;k}$ in~(\ref{eq:ML4par_rvproduct_density_infseries}) is
\begin{align*}
\Gamma(\alpha k+1)\,  \mu_{\alpha,\theta\vert\beta,\gamma;k} &= \Gamma(\alpha k+1)
 \frac{\Gamma(\beta+\theta)\Gamma(\gamma+\theta/\alpha+k)}{\Gamma(\gamma+\theta/\alpha)\Gamma(\alpha k+\beta+\theta)}  
 = \frac{\Gamma(\gamma+\theta/\alpha+k)}{\Gamma(\gamma+\theta/\alpha)}
\end{align*}
Hence~(\ref{eq:ML4par_rvproduct_density_infseries}) becomes
\begin{align*}
h_{\alpha,\theta; \alpha}(t \vert z , 1-\theta,\gamma)
&=   \frac{1}{\pi t}\,  {\rm Im} \sum_{k=0}^\infty  \frac{\Gamma(\gamma+\theta/\alpha+k)}{k! \, \Gamma(\gamma+\theta/\alpha)}
(-z e^{-i\pi\alpha} t^{-\alpha})^k \\
&=  \frac{1}{\pi t}\,  {\rm Im} \, \frac{1}{(1+z e^{-i\pi\alpha}t^{-\alpha})^{\gamma+\theta/\alpha}} \\
&=   \frac{1}{\pi}\,  {\rm Im} \,  \frac{t^{\alpha\gamma+\theta-1}}{(z e^{-i\pi\alpha}+t^{\alpha})^{\gamma+\theta/\alpha}}
 \end{align*}
 thereby proving~(\ref{eq:Lamperti-type_gen_density_infseries}).
 The rest of Proposition~$\ref{prop:Lamperti-type_gen}$ follows straightforwardly.
\end{proof}

\section{Discussion and Conclusion}
\label{sec:discussion}

We have  solely used convolution and mixing induced  by products of random variables to derive a rich family of distributions.
A common feature  is the presence  of the positive stable distribution in  both constructs.
The highlight of the paper  is the derivation of  a novel and general family of distributions 
with explicit  integral and series representations for their  densities.  
They include  the familiar Mittag-Leffler  and Lamperti-type laws as special cases.
 We  have had  no cause to mention the Dirichlet and Poisson-Dirichlet distributions, or random selection metaphors such as
  P{\' o}lya urns or restaurant seating.
 
The  family of distributions  has Laplace transforms
$\Gamma(\beta+\theta)\, E^{\gamma+\theta/\alpha}_{\alpha,\beta+\theta}(-z x^\sigma)$.
This common form of Laplace transforms suggests referring to the whole family of  distributions as Mittag-Leffler  distributions of various levels of generality.
We have constructed all distributions directly through mixing and convolution without invoking any complex analytic inversion.
Such inversion, where it arises implicitly, is  contained in the  derivation by Pollard~\cite{Pollard} of  representation~(\ref{eq:Pollard}) 
of the stable distribution.

\appendix
\numberwithin{equation}{section}

\section{Mittag-Leffler Functions}
\label{sec:MLfunction}

Gorenflo {\it et al.}~\cite{gorenflo2014mittag} is a comprehensive treatise on Mittag-Leffler functions and their applications.
The 3-parameter Mittag-Leffler function, also known as the Prabhakar function,
has an infinite series representation for $z\in \mathbb{C}$
\begin{align}
E^\gamma_{\alpha,\beta}(z) &= \frac{1}{\Gamma(\gamma)} 
   \sum_{k=0}^\infty \frac{\Gamma(\gamma+k)}{k!\,\Gamma(\alpha k+\beta)}\, z^k  \qquad {\rm Re}(\alpha)>0, \, {\rm Re}(\beta)>0, \, \gamma > 0
\label{eq:ML3parseries}
\end{align}
The 1-parameter Mittag-Leffler function $E_\alpha(z)$ is the case $\beta=\gamma=1$.
The termwise Laplace transform of $x^{\beta-1}E^\gamma_{\alpha,\beta}(-\lambda x^\alpha)$  $(\lambda>0)$ sums to
\begin{align}
\int_0^\infty e^{-sx} \, x^{\beta-1}E^\gamma_{\alpha,\beta}(-\lambda x^\alpha) \,dx  
  &= \frac{s^{\alpha\gamma-\beta}}{(\lambda+s^\alpha)^\gamma} 
\label{eq:ML3parseriesLT}
\end{align}
There are  many other attributes of Mittag-Leffler functions, often related to
fractional  calculus and its applications, but these  are not  of direct interest here.

What is key  in this paper is that we  often interpret  $E_\alpha(-y x^\alpha)$ 
as a function of two variables $x$ and $y$, not merely as  a function of $x$ for some fixed parameter $y$.
Thus, we    typically work with $E_\alpha(-y x^\alpha)$ as a  function of $x$ given $y$ or $y$ given $x$.

\bibliography{./IntegralRepresentation}{}
\bibliographystyle{plain} 
\end{document}